\documentclass[fleqn]{article}

\usepackage{amssymb, amsmath, amsfonts, amsthm, hyperref, mathtools}

\usepackage[english]{babel}

\newtheorem{theorem}{Theorem}[section]
\newtheorem{lemma}[theorem]{Lemma}
\newtheorem{corollary}[theorem]{Corollary}

\title{Biaxial escape in nematics at low temperature}

\author{Andres Contreras
\thanks{
Department of Mathematics and Statistics, McMaster University, Hamilton, Ontario, Canada L8S 4K1, contrera@math.mcmaster.ca
}
\and
Xavier Lamy
\thanks{
Universit\'e de Lyon, Institut Camille Jordan,  CNRS UMR 5208, Universit\'e Lyon 1, 43 blvd. du 11 novembre 1918, F-69622 Villeurbanne cedex, France, xlamy@math.univ-lyon1.fr
}
}

\begin{document}
\maketitle

\begin{abstract}
In the present work, we study minimizers of the Landau-de Gennes free
energy in a bounded domain $\Omega\subset \mathbb{R}^3$.  We prove
that at low temperature minimizers do not vanish, even for
topologically non-trivial boundary conditions. This is in contrast
with a  simplified Ginzburg-Landau model for superconductivity studied
by Bethuel, Brezis and H\'{e}lein. Merging this with an observation of Canevari we obtain, as a corollary, the occurence of \textit{biaxial escape}: the tensorial order parameter must become strongly biaxial at some point in  $\Omega$. In particular, while it is known that minimizers cannot be purely uniaxial, we prove the much stronger and physically relevant fact that they lie in a different homotopy class.
\end{abstract}

\section{Introduction}
\label{intro}

Nematic liquid crystals are composed of rigid rod-like molecules which tend to align in a preferred direction. As a result of this orientational order, nematics present electromagnetic properties similar to those of crystals. A striking feature of nematics  is the appearance of particular optical textures called \textit{defects}. From the mathematical point of view, the study of these defects is carried out using a tensorial order parameter $Q$ (introduced by P.G. de Gennes \cite{degennes}). The $Q$-tensor takes values in the five-dimensional space
\begin{equation}\label{Qspace}
\mathcal S = \left\lbrace Q\in \mathbb R^{3\times 3}\colon Q_{ij}=Q_{ji},\: \mathrm{tr}\: Q = 0 \right\rbrace,
\end{equation} 
of symmetric traceless $3\times 3$ matrices.
As a symmetric matrix, a $Q$-tensor  has an orthonormal frame of
eigenvectors: the eigendirections are the locally preferred mean
directions of alignment of the molecules, and the eigenvalues measure
the degrees of alignment along those directions. In this context, {\it uniaxial} states are described by $Q$-tensors with two equal eigenvalues, and {\it biaxial} states correspond to $Q$-tensors with three distinct eigenvalues. 

The configuration of a nematic material contained in a domain $\Omega\subset\mathbb R^3$ is given by a map $Q\colon \Omega \to\mathcal S$. At equilibrium, $Q$ should minimize the Landau-de Gennes free energy given by
\begin{equation}\label{F}
F_T(Q)=\int_\Omega \left( \frac L2 |\nabla Q|^2 + f_T(Q) \right) dx.
\end{equation}
Here $L$ is an elastic constant and $f_T (Q)$ is the bulk free energy density, usually considered to be of the form 
\begin{equation}\label{bulk}
f_T(Q)=\frac{\alpha(T-T_*)}{2} |Q|^2 - \frac b3 \mathrm{tr}(Q^3) + \frac c4 |Q|^4.
\end{equation}
Above $\alpha$, $b$ and $c$ are material-dependent positive constants, $T$ is the absolute temperature and $T_*$ a critical temperature. For $T<T_*$, the bulk free energy density $f_T(Q)$ attains its minimum exactly on the vacuum manifold $\mathcal N_T\subset\mathcal S$ composed of uniaxial $Q$-tensors with a certain fixed norm:
\begin{equation}\label{N}
\begin{gathered}
\mathcal N_T =\left\lbrace Q\in\mathcal S\colon Q=s_*\left(n\otimes n -\frac 13 I\right),\: n\in\mathbb S^2 \right\rbrace,\\
 s_*= s_*(T) = \frac {b+\sqrt{b^2-24 \alpha (T-T_*)c}}{4c}.
\end{gathered}
\end{equation}
Above, the notation $n\otimes n$ denotes the matrix $(n_in_j)$.
Note that $\mathcal N_T$ is diffeomorphic to the projective plane $\mathbb R \mathbb P ^2$. In this work we consider minimizers of  $F_T(Q)$ subject to Dirichlet boundary conditions $Q_{b,T}\colon\partial\Omega\to\mathcal N_T$ minimizing the potential $f_T(Q)$:
\begin{equation}\label{bcond}
Q_{b,T}(x)=s_*\left(n_b(x)\otimes n_b(x)-\frac 13 I\right),\quad  n_b\colon \partial\Omega\to\mathbb S^2.
\end{equation}
In the London limit $L\to 0,$ a minimizing $Q$-tensor must be close to an $\mathcal N_T$-valued harmonic map $Q_*$, that is a minimizer of the Dirichlet energy among maps with values in the manifold $\mathcal N_T$.  This is analogous to the case of the simplified Ginzburg-Landau energy with prescribed topologically nontrivial boundary conditions studied in \cite{BBH}; in this setting it is proved that minimizers of the corresponding energy converge to harmonic maps with values in $\mathbb S^1,$ which are then forced to have singularities, known in that context  as vortices.

 The singularities of the director field $n_*$  associated to the limit of minimizers of $F_T(Q)$ correspond to the optical defects observed in experiments. In the core of a defect, two possible behaviors are considered in the physics literature. The notion of {\it  isotropic melting} refers to a $Q$-tensor  vanishing in the core of the defect. This is comparable to the behaviour observed in the core of Ginzburg-Landau vortices, and can be achieved by remaining in a uniaxial state. Alternatively, $Q$-tensors may take advantage of the additional degrees of freedom offered by biaxiality: instead of vanishing in the core of the defect, the $Q$-tensor order parameter may become strongly biaxial. This last behaviour is referred to as {\it biaxial escape} \cite{sonnetkilianhess95}.

Biaxial escape has been first proposed as a way to avoid singularities of the director field by Lyuksyutov \cite{lyuksyutov78}. The corresponding mechanism has been investigated in greater detail by Penzenstadler and Trebin \cite{penzenstadlertrebin89}, followed by a number of further studies (see e.g. \cite{sonnetkilianhess95,rossovirga96,gartlandmkaddem00,kraljvirga01} ). These works indicate that biaxial escape should be energetically favorable when the bulk free energy \eqref{bulk} degenerates to a Ginzburg-Landau-like potential, which occurs for instance at low temperature.

Our main result states that, at low temperatures, isotropic melting is indeed avoided: the minimizing configurations do not vanish. 

\begin{theorem}\label{nomelting}
Let $\Omega\subset\mathbb R^3$ be a smooth bounded simply connected domain. Let $n_b\colon\partial\Omega\to\mathbb S^2$ be a smooth director field and $Q_{b,T}\colon\partial\Omega\to\mathcal N_T$ the associated boundary datum \eqref{bcond}. Let $Q_T$ be a solution of the variational problem
\begin{equation*}
\min \left\lbrace F_T(Q) \colon Q\in H^1(\Omega ; \mathcal S),\: Q=Q_{b,T}\text{ on }\partial\Omega\right\rbrace,
\end{equation*}
where $F_T$ is the Landau-de Gennes free energy \eqref{F}. Then, there exists $T_0\in\mathbb R$ (depending on $\Omega,L,\alpha,b,c$), such that if $T<T_0$, 
$$
\inf_\Omega |Q_T| > 0,
$$
i.e. $Q_T$ does not vanish in $\Omega .$
\end{theorem}

To prove Theorem~\ref{nomelting}, we use the fact that any zero $x_T$  of $Q_T$ must converge, as $T\to -\infty ,$ to a point $x_0 \in \Omega ;$ this follows from the analysis in \cite{majumdarzarnescu10}. After this, we take advantage of the degeneracy of the bulk potential to a Ginzburg-Landau potential in the low temperature limit. The Ginzburg-Landau potential being minimized by $\mathbb S^4$-valued maps, we are able to relate $Q_T$ to an $\mathbb S^4$-valued harmonic map. This is done through a blow-up analysis of $Q_T$ at $x_T$ which  in turn leads to a local minimization problem in $\mathbb R^3 $ for a limiting map $Q_\infty$. Next, thanks to the study in \cite{MillPis} based on the work of Lin and Wang \cite{LinWang}, a  blow-down analysis of the limiting map using the minimality of $Q_\infty$ yields strong convergence to a harmonic map with values in $\mathbb{S}^4.$ The conclusion follows with the help of a regularity result for minimizing harmonic maps by Schoen and Uhlenbeck \cite{schoenuhlenbeck84}.

Next we explain how Theorem~\ref{nomelting} is related to the phenomenon of biaxial escape. Of course, Theorem~\ref{nomelting} is more interesting when the boundary condition $n_b$ is topologically non-trivial. In that case, a recent remark of Canevari \cite[Lemma~3.10]{canevari14} shows that the only way for $Q_T$ to avoid isotropic melting is to be {\it strongly} biaxial. To give a precise meaning to this statement, we recall the definition of the biaxiality parameter for a $Q$-tensor,
\begin{equation}\label{beta}
\beta (Q) = 1-6 \frac{(\mathrm{tr}(Q^3))^2}{|Q|^6},
\end{equation}
introduced in \cite{kaiserwiesehess92}. It holds that
$
0\leq\beta(Q)\leq 1,
$
and $Q$ is uniaxial for $\beta=0$, biaxial for $\beta>0$ and is said to be \textit{maximally biaxial} for $\beta=1$. Canevari's lemma implies the following corollary to our main result:

\begin{corollary}\label{biax}
If the boundary datum $n_b\colon\partial\Omega\to\mathbb S^2$ is topologically non-trivial, then for low enough temperatures $T<T_0$, any minimizing configuration $Q_T$ must be strongly biaxial:
\begin{equation*}
\beta(Q_T(x_0)) = 1
\end{equation*}
for some $x_0\in\Omega$.
\end{corollary}

In fact, in \cite{canevari14}  Canevari uses the aforementioned lemma to prove a theorem similar to Corollary~\ref{biax}, in the case  of a two-dimensional domain. Our result is a three-dimensional analog of \cite[Theorem~1.1]{canevari14}, and could probably be adapted to provide a simpler proof of \cite[Theorem~1.1]{canevari14}.

Corollary~\ref{biax} generalizes a recent result by Henao,
Majumdar and Pisante \cite{henaomajumdarpisante14}. In
\cite{henaomajumdarpisante14}, the authors show that for low enough
temperature, minimizers can not be purely uniaxial (that is, can not
satisfy $\beta=0$ everywhere). Note that such a result does not
exclude the existence of approximately uniaxial minimizers, while
Corollary~\ref{biax} does. Moreover the results of the second author
in \cite{lamy14} indicate that the uniaxiality constraint is very
rigid: non-existence of purely uniaxial solutions may not be
specific to low temperature or energy minimization.
 In contrast, Corollary~\ref{biax} is really specific to the low temperature limit.

The article is organized as follows. 
In Section~\ref{spropmin} we reformulate the problem and recall some basic convergence properties of minimizers of $F_T$. 
In Section~\ref{sblowup} we study the blown-up problem, obtain a limiting map and derive its minimal character. 
In Section~\ref{sblowdown} we conclude the proof of Theorem~\ref{nomelting} with the aid of a blow-down analysis. 
Finally, in Section~\ref{sproofbiax} we prove Corollary~\ref{biax} and make some final remarks.

\subsection*{Acknowledgements} AC would like to thank his postdoctoral
supervisors S.~Alama and L.~Bronsard for their support and
encouragement. He was funded by McMaster's postdoctoral fellowship. This work was carried out during XL's
visit at McMaster University. XL thanks the Department of Mathematics
and Statistics, in  particular S.~Alama and L.~Bronsard, for their
hospitality, and the `Programme Avenir Lyon Saint-Etienne' for its
financial support. He also wishes to thank his Ph.D. advisor
P.~Mironescu for his constant support and helpful advice.

\section{Properties of minimizing Q-tensors}\label{spropmin}

\subsection{Rescaling}
\label{rescal}

 Introducing the reduced temperature $t$ and rescaled maps $\widetilde Q$: 
\begin{equation*}
t:=\frac{-\alpha (T-T_*) c}{b^2},\quad \widetilde Q := \frac{1}{s_*}\sqrt{\frac 32}Q,
\end{equation*}
 we see that, for some constant $K=K(\alpha,b,c,T)$ which plays no role in the sequel,
\begin{equation*}
F_T(Q)=\frac{s_*^2 b^2}{3c}\int_\Omega\left(\frac{\widetilde L}{2}|\nabla \widetilde Q|^2 + \frac{t}{2} (|\widetilde Q|^2 - 1)^2 + \lambda(t)\,h(\widetilde Q ) \right) dx \ + K ,
\end{equation*}
where $\widetilde L = 3cL/b^2$,
\begin{equation}\label{lambda}
\lambda(t) = \frac{\sqrt{24t+1}+1}{12 } \underset{t\to +\infty}{\sim} \sqrt{\frac t6},
\end{equation}  
and
\begin{equation}\label{h}
h(\widetilde Q)=  \frac{1}{6}-\frac{2\sqrt 2}{\sqrt 3}\mathrm{tr}(\widetilde Q^3)+ \frac 12 |\widetilde Q|^4.
\end{equation}
It holds that $h(Q)\geq 0$ for every $Q\in\mathcal S$, 
and the potential $h$ vanishes exactly at
\begin{equation}\label{tildeN}
\widetilde{\mathcal N} =\left\lbrace \sqrt{\frac 32}\left( n\otimes n-\frac 13 I\right) \colon n\in\mathbb S^2 \right\rbrace.
\end{equation}

The limit $T\to -\infty$ corresponds to $t\to +\infty$.
Therefore we may reformulate the problem:  show that minimizers $Q_t$ of the energy functional
\begin{equation}\label{tildeFt}
\widetilde F_t (Q) = \int_\Omega \left(\frac{\widetilde L}{2}|\nabla Q|^2 + \frac t2 (| Q|^2 - 1)^2 + \lambda(t)\, h( Q) \right)\, dx
\end{equation}
subject to the boundary condition
\begin{equation}\label{tildeQb}
Q_t=\widetilde Q_b = \sqrt{\frac 32} \left( n_b \otimes n_b -\frac 13 I \right)\quad\text{on }\partial\Omega,
\end{equation}
do not vanish for large enough $t$. 

We prove Theorem~\ref{nomelting} by contradiction: we assume the existence of  sequences $t_j\to + \infty$ and $(x_j)\subset\Omega$ such that $Q_{t_j}$ minimizes \eqref{tildeFt}-\eqref{tildeQb} and $Q_{t_j}(x_j)=0$. Note that any minimizer of $\widetilde F_t$ is smooth thanks to standard elliptic estimates (see e.g. \cite[Proposition~13]{majumdarzarnescu10}), so that evaluation at $x_j$ makes sense. Up to extracting a subsequence, we may assume in addition that $x_j\to x_*\in\overline\Omega$. 

In the sequel we study the behaviour of the sequence $(Q_{t_j})$ and obtain a contradiction. To simplify the notations, we drop the subscript $j$: we write $(Q_t)$ and $(x_t)$ and it is always implied that a subsequence is considered.

\subsection{Convergence}\label{conv}

Since the set $H^1_{n_b}(\Omega;\mathbb S^2)=\lbrace n\in H^1(\Omega;\mathbb S^2)\colon n_{|\partial\Omega}=n_b\rbrace$ is not empty, we may use an $\widetilde{\mathcal N}$-valued comparison map and obtain the bound
\begin{equation}\label{boundFt}
\widetilde F_{t} (Q_{t}) =\int_\Omega \left(\frac{\widetilde L}{2}|\nabla Q_t|^2 + \frac {t}{2} (| Q_{t}|^2 - 1)^2 + \lambda(t)\, h( Q_{t}) \right)\, dx\leq C.
\end{equation}
In particular, we see that the sequence $(Q_{t})$ is bounded in $H^1(\Omega;\mathcal S)$. Up to extracting a subsequence, we may therefore assume that $Q_{t}$ converges weakly to a limiting map $Q_*\in H^1(\Omega;\mathcal S)$. Moreover, since the bound \eqref{boundFt} implies
\begin{equation*}
\int_\Omega h(Q_{t}) \leq C\lambda(t)^{-1} \sim C\sqrt{\frac {6}{t}},
\end{equation*}
we deduce that $h(Q_*)=0$ a.e., so that $Q_*$ is $\widetilde{\mathcal N}$-valued. From this point on, we can proceed exactly as in \cite[Lemma~3]{majumdarzarnescu10}. We conclude that $Q_{t}$ converges to $Q_*$ strongly in $H^1$, and that
\begin{equation*}
Q_* = \sqrt{\frac 32} \left( n_* \otimes n_* -\frac 13 I \right),
\end{equation*}
where $n_*\in H^1(\Omega;\mathbb S^2)$ is a minimizing harmonic map. In particular, $Q_*$ is smooth in $\Omega\setminus \Sigma$, where $\Sigma\subseteq \Omega$ is a finite set of interior point singularities \cite{schoenuhlenbeck82,schoenuhlenbeck83}.

As in the Ginzburg-Landau case \cite{bbh93}, the convergence of
$Q_{t}$ towards $Q_*$ can be improved away from the singularities
$\Sigma$. The arguments in \cite{bbh93} have been adapted to the liquid
crystal case in \cite{majumdarzarnescu10}. The asymptotic regime $L\to
0$ in \cite{majumdarzarnescu10} corresponds to the 
limit $t\to + \infty $ in the present work.
 The arguments in \cite[Proposition~4]{majumdarzarnescu10} and
 \cite[Proposition~6]{majumdarzarnescu10} are
 straightforward to adapt, and we obtain the convergence
\begin{equation*}
\frac {1}{2} (| Q_{t}|^2 - 1)^2 + \frac{\lambda(t)}{t}\, h(
Q_{t})\longrightarrow 0,\quad\text{locally uniformly in }\overline
\Omega\setminus \Sigma.
\end{equation*}
Since we have in addition, thanks to the maximum principle, $|Q_{t}|\leq 1$ (cf e.g. \cite[Proposition~3]{majumdarzarnescu10} ), we deduce -- using also \eqref{lambda} -- that 
\begin{equation}\label{convQt}
|Q_{t}|\longrightarrow 1\quad\text{locally uniformly in }\overline \Omega\setminus \Sigma.
\end{equation}

Recall that by assumption, $Q_{t}(x_t)=0$ for a sequence $x_t\to
x_*\in\overline\Omega$. The uniform convergence
\eqref{convQt}
 away from $\Sigma$ implies that $x_*\in \Sigma$. In particular $x_*$ lies well inside $\Omega$. Our next step will consist in ``blowing up'' around $x_t$.

\section{Blowing up}\label{sblowup}

We  fix $\delta>0$ such that $B(x_t,\delta)\subset\Omega$ for all $j$. We consider the blown-up maps
\begin{equation*}
\overline Q_{t} (x) = Q_{t}\left( x_t + \frac{x}{\sqrt t} \right),\quad x\in B_{\delta\sqrt {t}}.
\end{equation*}
The map $\overline Q_{t}$ minimizes the energy functional
\begin{equation}\label{Et}
E_{t}(Q;B_{\delta\sqrt {t}}) = \int_{B_{\delta\sqrt {t}}}\left(\frac{\widetilde L}{2}|\nabla Q|^2 + \frac 12 (|Q|^2-1)^2 \right)\, dx + \frac{\lambda(t)}{t}\int_{B_{\delta\sqrt {t}}}h(Q)\, dx,
\end{equation}
with respect to its own boundary conditions. Fix any $R>0$. For large enough $t$, $\overline Q_t$ is defined in $B_R$ and solves the Euler-Lagrange equation
\begin{equation*}
\widetilde L \Delta \overline Q_t = 2(|\overline Q_t|^2-1)\overline Q_t + \frac{\lambda(t)}{t} \nabla h(\overline Q_t).
\end{equation*}
The uniform bound $|\overline Q_t|\leq 1$ and standard elliptic estimates thus imply
\begin{equation*}
|\nabla \overline Q_t |\leq C_R\quad\text{in }B_R,
\end{equation*}
where $C_R$ is a constant that may depend on $R$ but not on $t .$
Therefore, up to extracting a subsequence, we may assume that $\overline Q_t$ converges locally uniformly, and weakly in $H^1_{\mathrm{loc}}$, to a map $Q_\infty\in H^1_{\mathrm{loc}}(\mathbb R^3;\mathcal S)$. Moreover, since the convergence is locally uniform, $Q_\infty$ is continuous and satisfies
\begin{equation}\label{Qinfty0}
Q_\infty(0)=0.
\end{equation}

We claim that $Q_\infty$ locally minimizes a Ginzburg-Landau energy; this is a very important simplification.

\begin{lemma}\label{localminimality}For all $R>0 ,$ the limiting profile $Q_\infty$ minimizes the energy functional
\begin{equation}\label{E}
E(Q;B_R) = \int_{B_R} \left( \frac{\widetilde L}{2}|\nabla Q|^2 + \frac 12 (|Q|^2-1)^2 \right)\, dx ,
\end{equation}
with respect to its own boundary condition.
\end{lemma}

\begin{proof}
 Let $P\in H_0^1(B_R;\mathcal S)$. Since $\overline Q_t$ is minimizing, it holds
\begin{align*}
0 & \leq E_t(\overline Q_t + P; B_R) - E_t(\overline Q_t; B_R) \\
& = \widetilde{L}\int_{B_R} \nabla \overline Q_t \cdot\nabla P + \frac{\widetilde L}{2}\int_{B_R} |\nabla P|^2 \\
&\quad + \frac{1}{2}\int_{B_R} (|\overline Q_t+P|^2-1)^2 -\frac 12 \int_{B_R} (|\overline Q_t|^2-1)^2 \\
& \quad + \frac{\lambda(t)}{t}\int_{B_R} \big[ h(\overline Q_t+P)-h(\overline Q_t) \big] \, dx.
\end{align*}
Using the weak $H^1$ convergence of $\overline Q_t$ (which implies also strong $L^6$ convergence), we obtain in the limit $t\to + \infty$
\begin{align*}
0 & \leq \widetilde L \int_{B_R} \nabla Q_\infty \cdot\nabla P \, dx + \frac{\widetilde L}{2}\int_{B_R} |\nabla P|^2 \\
&\quad + \frac{1}{2}\int_{B_R} (|Q_\infty+P|^2-1)^2 -\frac 12 \int_{B_R} (|Q_\infty|^2-1)^2 \\
& = E(Q_\infty + P;B_R)-E(Q_\infty;B_R).
\end{align*}
Therefore $Q_\infty$ minimizes \eqref{E}, as claimed. 
\end{proof}


Moreover, proceeding exactly as in the proof of \cite[Theorem~1.(v)]{henaomajumdarpisante14}, we obtain the energy bound
\begin{equation}\label{boundQinfty}
E(Q_\infty ; B_R) \leq C R.
\end{equation}
The bound \eqref{boundQinfty} follows from two main ingredients: an  energy monotonicity inequality for minimizers of \eqref{Et} \cite[Lemma~2]{majumdarzarnescu10}, and an energy bound for $\mathbb S^2$-valued  minimizing harmonic maps near their singularities (following from the energy monotonicity for minimizing harmonic maps, see e.g. \cite[Lemma~2.2.5]{wanglin08}).

\section{Blowing down}\label{sblowdown}

Our last step consists in ``blowing down'' $Q_\infty$ around the origin, and eventually reaching a contradiction with \eqref{Qinfty0}.
Let $B_1$ be the unit ball in $\mathbb{R}^3.$ We consider the blown-down maps
\begin{equation*}
\underline Q_R (x) = Q_\infty(Rx),\quad x\in B_1.
\end{equation*}
Note that \eqref{Qinfty0} implies that
\begin{equation}\label{blowdown0}
\underline Q_R(0)=0, \quad\forall R>0.
\end{equation}
By definition, $\underline Q_R \in H^1(B_1)$ for all $R>0$. We have:

\begin{lemma}\label{blowdown}
Up to a subsequence,
\begin{equation*}
\underline Q_R\longrightarrow \underline Q \quad\text{ in }
H^1(B_1;\mathcal S ),
\end{equation*}
for some $\mathbb{S}^4$-valued harmonic map $\underline Q.$  Moreover,
$|\underline Q_R|$ stays bounded away from zero uniformly in $B_1 .$

\end{lemma}

\begin{proof}
Since $Q_\infty$ minimizes \eqref{E}, the map $\underline Q_R$ minimizes the energy functional
\begin{equation}\label{G}
G_R(Q) = \int_{B_1} \left( \frac{\widetilde L}{2}|\nabla Q|^2 + \frac{R^2}{2}(|Q|^2-1)^2 \right) \, dx.
\end{equation}

Moreover, the energy bound \eqref{boundQinfty} implies the bound
\begin{equation}\label{boundblowdown}
G_R( \underline Q_R) \leq C,
\end{equation}
so that we may extract a subsequence $R\to +\infty$ (indices are implicit), such  that
\begin{equation}\label{convblowdown1}
\underline Q_{R}\longrightarrow \underline Q \quad\text{weakly in }H^1(B_1; \mathcal S).
\end{equation}
The energy bound \eqref{boundblowdown} also implies that $\underline Q$ is $\mathbb S^4$-valued. 
Now, thanks to Lemma~\ref{localminimality}, we can appeal to Proposition  4.2 in \cite{MillPis} to conclude that the convergence of $\underline Q_R$ to $\underline Q$ can be improved to strong convergence in $H^1$. In \cite{MillPis}, the proof relies on \cite[Theorem~C]{LinWang} in the case of $\mathbb R^3$-valued maps converging to $\mathbb S^2$-valued maps. However, \cite[Theorem~C]{LinWang} is valid in greater generality and applies to our case. Moreover, the analysis in \cite{MillPis} does not make use of the dimension of the target space other than  to provide an explicit constant in their computations. 

Next, the minimizing character of $\underline Q$ follows from Step 1 in \cite[Corollary 4.1]{MillPis}, which also applies to our case without modifications. 
From this we conclude that $\underline Q$ is an $\mathbb{S}^4$-valued minimizing harmonic map. 
As a consequence,
Schoen and Uhlenbeck's regularity result \cite[Theorem~2.7]{schoenuhlenbeck84} ensures that $\underline Q$ is smooth in $B_1$. 

Since  the proof of \cite[Proposition~4.2]{MillPis} also shows that the convergence of $\underline Q_R$ towards $\underline Q$ is actually uniform away from the singularities of $\underline Q$, we obtain in particular that
\begin{equation}\label{asympS4}
|\underline Q_R|\longrightarrow 1 \quad\text{uniformly in }B_1,
\end{equation}
which is the desired conclusion.
\end{proof}

We note that \eqref{asympS4} contradicts \eqref{blowdown0} and thus the proof of Theorem~\ref{nomelting} is complete. \qed

\section{Proof of Corollary~\ref{biax}}\label{sproofbiax}

 In \cite{canevari14}, Canevari makes the crucial observation that if $Q$ is {almost uniaxial}, i.e. 
\begin{equation*}
\max_{\overline \Omega} \beta (Q) <1,
\end{equation*}
then the $Q$-tensor must vanish. More precisely, in our case the following result holds.

\begin{lemma}{\cite[Lemma~3.10]{canevari14}} \label{canevaricone}Let $Q\in C^1(\overline \Omega; \mathcal S)$ with uniaxial boundary condition of the form \eqref{bcond}. If $n_b\colon\partial\Omega\to\mathbb S^2$ is topologically non trivial, and $Q$ is almost uniaxial, then
\begin{equation*}
\min_{\overline \Omega} |Q|=0.
\end{equation*} 
\end{lemma}

In \cite{canevari14} the proof is carried out in the two-dimensional case but a careful reading shows that the argument still holds in the three-dimensional setting, since the result relies only on topological considerations in the target space $\mathcal S$. Indeed, the crucial observation leading to \cite[Lemma~3.10]{canevari14} is the fact that, for any $C\geq 1$ and $1>\delta>0$, the set
\begin{equation*}
\left\lbrace Q\in\mathcal S\colon \delta \leq |Q| \leq C,\; \beta(Q)\leq 1-\delta \right\rbrace\subset\mathcal S
\end{equation*}
is topologically equivalent to $\mathcal N\simeq \mathbb{R P}^2$.

As a consequence of Theorem~\ref{nomelting} we see that, in light of Lemma~\ref{canevaricone}, $Q_T$ must be maximally biaxial at some point for sufficiently low temperature. The proof of Corollary~\ref{biax} is complete.\qed

We finish with a few remarks. Theorem~\ref{nomelting} implies the existence of a point where maximal biaxiality is achieved, however it does not provide a characterization of the location of this (or these) point(s) in terms of the domain or the boundary datum. Also the number of  these points of biaxial escape cannot be deduced from the topological conclusion in 
\cite[Lemma~3.10]{canevari14}. To finish, a more detailed description of the defect core is also an interesting matter worthy of pursuit. In this last direction, we mention the stability study of the radial hedgehog defect performed in \cite{Igetal}.


\bibliographystyle{plain}
\bibliography{biax}

\end{document}